\documentclass[reqno,a4paper]{amsart}
\usepackage{amssymb,setspace}
\usepackage{ifpdf}
\ifpdf
 \usepackage[hyperindex,pagebackref]{hyperref}
\else
 \expandafter\ifx\csname dvipdfm\endcsname\relax
 \usepackage[hypertex,hyperindex,pagebackref]{hyperref}
 \else
 \usepackage[dvipdfm,hyperindex,pagebackref]{hyperref}
 \fi
\fi
\allowdisplaybreaks[4]
\numberwithin{equation}{section}
\theoremstyle{plain}
\newtheorem{thm}{Theorem}[section]
\theoremstyle{remark}

\DeclareMathOperator{\td}{d}

\begin{document}

\title[Asymptotic formulas and inequalities for gamma function]
{Asymptotic formulas and inequalities for gamma function in terms of tri-gamma function}

\author[C. Mortici]{Cristinel Mortici}
\address[Mortici]{Department of Mathematics, Valahia University of T\^argovi\c{s}te, Bd. Unirii 18, 130082 T\^argovi\c{s}te, Romania}
\email{\href{mailto: C. Mortici <cmortici@valahia.ro>}{cmortici@valahia.ro}}
\urladdr{\url{http://www.cristinelmortici.ro}}

\author[F. Qi]{Feng Qi}
\address[Qi]{College of Mathematics, Inner Mongolia University for Nationalities, Tongliao City, Inner Mongolia Autonomous Region, 028043, China; Department of Mathematics, College of Science, Tianjin Polytechnic University, Tianjin City, 300387, China; Institute of Mathematics, Henan Polytechnic University, Jiaozuo City, Henan Province, 454010, China}
\email{\href{mailto: F. Qi <qifeng618@gmail.com>}{qifeng618@gmail.com}, \href{mailto: F. Qi <qifeng618@hotmail.com>}{qifeng618@hotmail.com}, \href{mailto: F. Qi <qifeng618@qq.com>}{qifeng618@qq.com}}
\urladdr{\url{http://qifeng618.wordpress.com}}

\begin{abstract}
In the paper, the authors establish some asymptotic formulas and double inequalities for the factorial $n!$ and the gamma function $\Gamma$ in terms of the tri-gamma function $\psi'$.
\end{abstract}



\thanks{This paper was typeset using \AmS-\LaTeX}

\maketitle

\section{Introduction and motivation}

We recall that the classical Euler's gamma function may be defined by
\begin{equation}\label{gamma-dfn}
\Gamma(z)=\int^\infty_0t^{z-1} e^{-t}\td t
\end{equation}
for $\Re(z)>0$, that the logarithmic derivative of $\Gamma(x)$ is called psi or di-gamma function and denoted by
\begin{equation}
  \psi(x)=\frac{\td}{\td x}\ln\Gamma(x)=\frac{\Gamma'(x)}{\Gamma(x)}
\end{equation}
for $x>0$, that the derivatives $\psi'(x)$ and $\psi''(x)$ for $x>0$ are respectively called tri-gamma and tetra-gamma functions, and that the derivatives $\psi^{(i)}(x)$ for $i\in\mathbb{N}$ and $x>0$ are called polygamma functions.
\par
We also recall from~\cite[Chapter~XIII]{mpf-1993} and~\cite[Chapter~IV]{widder} that a function $f(x)$ is said to be completely monotonic on an interval $I$ if it has derivatives of all orders on $I$ and satisfies $0\le(-1)^{n}f^{(n)}(x)<\infty$ for $x\in I$ and all integers $n\ge0$. The class of completely monotonic functions may be characterized by the celebrated Bernstein-Widder Theorem~\cite[p.~160, Theorem~12a]{widder} which reads that a necessary and sufficient condition that $f(x)$ should be completely monotonic in $0\le x<\infty$ is that
\begin{equation} \label{berstein-1}
f(x)=\int_0^\infty e^{-xt}\td\alpha(t),
\end{equation}
where $\alpha(t)$ is bounded and non-decreasing and the integral converges for $0\le x<\infty$.
\par
In~\cite[Theorem~2.1]{Sevli-Batir-Modelling}, it was proved that the function
\begin{equation}
F_{\alpha}(x) =\ln \Gamma(x+1) -x\ln x+x-\frac{1}{2}\ln x-\frac12\ln(2\pi)-\frac{1}{12}\psi'(x+\alpha)
\end{equation}
is completely monotonic on $(0,\infty)$ if and only if $\alpha\ge\frac12$ and that the function $-F_{\alpha}(x)$ is completely monotonic on $(0,\infty)$ if and only if $\alpha=0$. Consequently, the double inequality
\begin{equation}\label{double-ineqSevli-BatirModel}
\frac{x^x}{e^x}\sqrt{2\pi x}\,\exp \biggl(\frac{1}{12}\psi'\biggl(x+\frac{1}{2}\biggr) \biggr) <\Gamma(x+1) <\frac{x^x}{e^x}\sqrt{2\pi x}\,\exp \biggl(\frac{1}{12}\psi'(x) \biggr)
\end{equation}
was derived in~\cite[Corollary~2.1]{Sevli-Batir-Modelling}. These results were also established in~\cite{AMIS042013A.tex} and its preprint~\cite{Merkle-Convexity2Complete-Mon.tex} independently from a different origin and by a different motivation. For some more information on bounding the gamma function $\Gamma$, please refer to the newly published paper~\cite{Bukac-Sevli-Gamma.tex}, the survey articles~\cite{bounds-two-gammas.tex, Wendel2Elezovic.tex-JIA, Wendel-Gautschi-type-ineq-Banach.tex}, and plenty of references collected therein.
\par
The goal of this paper is to discover best asymptotic formulas and double inequalities for the factorial $n!=\Gamma(n+1)$ and the gamma function $\Gamma(x)$ in terms of the tri-gamma function $\psi'\bigl(x+\frac12\bigr)$. These results have something to do with the function $F_{\alpha}(x)$ and the double inequality~\eqref{double-ineqSevli-BatirModel}.

\section{An asymptotic formula and a double inequality for $n!$}

In this section, we establish a best asymptotic formula and a double inequality for the factorial $n!=\Gamma(n+1)$ in terms of the tri-gamma function $\psi'\bigl(x+\frac12\bigr)$.

\begin{thm}\label{Qi-Moticic-best-approx}
As $n\to \infty$, the asymptotic formula
\begin{equation}  \label{a}
n! \sim \frac{n^n}{e^n}\sqrt{2\pi n}\,\exp \biggl(\frac{1}{12}\psi'\biggl(n+\frac{1}{2}\biggr) \biggr)
\end{equation}
is the most accurate one among all approximations of the form
\begin{equation}\label{aa}
n! \sim \frac{n^n}{e^n}\sqrt{2\pi n}\,\exp \biggl(\frac{1}{12}\psi'(n+a)\biggr),
\end{equation}
where $a\in\mathbb{R}$.
\end{thm}

\begin{proof}
For $n\ge 1$, define a sequence $w_{n}$ by
\begin{equation*}
n!=\Gamma(n+1)=\sqrt{2\pi}\,n^{n+1/2}e^{-n}\exp \biggl(
\frac{1}{12}\psi'(n+a) \biggr) \exp w_{n}.
\end{equation*}
Taking into account
\begin{equation}\label{psi-poly-recur}
\psi^{(k)}(z+1)=\psi^{(k)}(z)+(-1)^k\frac{k!}{z^{k+1}}
\end{equation}
for $k=1$, see~\cite[p.~260, 6.4.6]{aa}, yields
\begin{equation*}
w_{n+1}-w_{n}=1+\ln(n+1)-\biggl(n+\frac{3}{2}\biggr) \ln(n+1) +\biggl(n+\frac{1}{2}\biggr) \ln n+\frac{1}{12(n+a)^{2}}
\end{equation*}
and
\begin{equation*}
w_{n+1}-w_{n}=\biggl(-\frac{1}{6}a+\frac{1}{12}\biggr) \frac{1}{n^{3}} +\biggl(\frac{1}{4}a^{2} -\frac{3}{40}\biggr) \frac{1}{n^{4}}+O\biggl(\frac{1}{n^{5}}\biggr).
\end{equation*}
Hence, we have
\begin{equation*}
\lim_{n\to\infty}\bigl\{n^{3}\bigl[w_{n+1}-w_{n}\bigr]\bigr\}=\frac{1}{12}-\frac{1}{6}a.
\end{equation*}
Lemma~1.1 in~\cite{m1, Mortici-aar.tex} states that if the sequence $\{\omega_n:n\in\mathbb{N}\}$ converges to $0$ and
\begin{equation}
\lim_{n\to\infty}n^k(\omega_n-\omega_{n+1})=\ell\in\mathbb{R}
\end{equation}
for $k>1$, then
\begin{equation}
\lim_{n\to\infty}n^{k-1}\omega_n=\frac{\ell}{k-1}.
\end{equation}
Consequently, the sequence $w_n$ converges fastest only if $a=\frac12$.
\end{proof}

\begin{thm}\label{ii-thm}
For every integer $n\ge 1$, we have
\begin{equation}\label{ii}
\exp \biggl(\frac{1}{240n^{3}}-\frac{11}{6720n^{5}}\biggr)
<\frac{e^nn!}{n^{n}\sqrt{2\pi n}\exp \bigl( \frac{1}{12}\psi'\bigl(n+\frac{1}{2}\bigr)\bigr)}
<\exp \frac{1}{240n^{3}}.
\end{equation}
\end{thm}

\begin{proof}
The double inequality~\eqref{ii} may be rewritten as
\begin{equation} \label{f}
f(n)=\ln \Gamma(n+1)-\biggl(n+\frac{1}{2}\biggr)
\ln n+n-\frac{1}{2}\ln (2\pi) -\frac{1}{12}\psi'\biggl(n+\frac{1}{2}\biggr) -\frac{1}{240n^{3}} \le0
\end{equation}
and
\begin{multline}\label{g}
g(n)=\ln \Gamma(n+1)-\biggl(n+\frac{1}{2}\biggr)\ln n+n-\frac{1}{2}\ln (2\pi) \\
-\frac{1}{12}\psi'\biggl(n+\frac{1}{2}\biggr) -\frac{1}{240n^{3}}+\frac{11}{6720n^{5}}
\ge0.
\end{multline}
Employing the recurrence formula~\eqref{psi-poly-recur} applied to $k=1$ and straightforwardly computing reveal that $f(n+1)-f(n)=u(n)$ and $g(n+1)-g(n)=v(n)$, where
\begin{align*}
u(x)  &=1+\ln(x+1) -\biggl(x+\frac{3}{2}\biggr) \ln(x+1) +\biggl(x+\frac{1}{2}\biggr) \ln x\\
&\quad+\frac{1}{12\bigl(x+\frac{1}{2}\bigr)^{2}} -\frac{1}{240(x+1)^{3}}+\frac{1}{240x^{3}}
\end{align*}
and
\begin{align*}
v(x)  &=1+\ln(x+1) -\biggl(x+\frac{3}{2}\biggr) \ln(x+1) +\biggl(x+\frac{1}{2}\biggr) \ln x+\frac{1}{12\bigl(x+\frac{1}{2}\bigr)^{2}} \\
&\quad-\frac{1}{240(x+1)^{3}}+\frac{1}{240x^{3}}+\frac{11}{6720(x+1)^{5}}-\frac{11}{6720x^{5}}.
\end{align*}
It is not difficult to verify that
\begin{equation*}
u''(x) = \frac{13x+74x^{2}+232x^{3}+391x^{4}+330x^{5}+110x^{6}+1}{20x^{5}(x+1)^{5}(2x+1)^{4}} >0
\end{equation*}
and
\begin{equation*}
v''(x) =-\frac{Q(x)}{1120x^{7}(x+1)^{7}(2x+1)^{4}}<0,
\end{equation*}
where
\begin{align*}
Q(x)  &=825x+5499x^{2}+21325x^{3}+52589x^{4} \\
&\quad+83867x^{5}+83881x^{6}+47936x^{7}+11984x^{8}+55.
\end{align*}
This shows that $u(x)$ is strictly convex and $v(x)$ is strictly concave on $(0,\infty)$. Further considering $\lim_{x\to\infty}u(x) =\lim_{x\to\infty}v(x) =0$, we obtain that $u(x)>0$ and $v(x)<0$ on $(0,\infty)$. Consequently, the sequence $f(n)$ is strictly increasing and $g(n)$ is strictly decreasing while they both converge to $0$. As a result, we conclude that $f(n)<0$ and $g(n)>0$ for every integer $n\ge 1$. The proof of Theorem~\ref{ii-thm} is complete.
\end{proof}

\section{An asymptotic series and a double inequality for $\Gamma$}

We now discover an asymptotic series and a double inequality for the gamma function $\Gamma(x)$ in terms of the tri-gamma function $\psi'\bigl(x+\frac12\bigr)$.

\begin{thm}\label{Gamma-asymp-thm}
As $x\to\infty$, we have
\begin{multline}
\Gamma(x+1)\sim\sqrt{2\pi}\,x^{x+1/2}\exp \biggl(\frac{1}{12}\psi'\biggl(x+\frac{1}{2}\biggr)-x+\frac{1}{240}\frac1{x^{3}}\\*
-\frac{11}{6720}\frac1{x^{5}}+\frac{107}{80640}\frac1{x^{7}} -\frac{2911}{1520640}\frac1{x^{9}}+\dotsm\biggr).
\end{multline}
\end{thm}

\begin{proof}
Motivated by the inequality~\eqref{ii}, we now consider a new function $h(x)$ defined by
\begin{equation*}
\Gamma(x+1) =\sqrt{2\pi}\,x^{x+1/2}e^{-x}\exp \biggl(
\frac{1}{12}\psi'\biggl(x+\frac{1}{2}\biggr) \biggr) \exp h(x) ,
\end{equation*}
that is,
\begin{equation*}
h(x) =\biggl[ \ln \Gamma(x+1) -\biggl(x+\frac{1}{2}\biggr) \ln x+x-\ln \sqrt{2\pi}\,\biggr] -\frac{1}{12}\psi'\biggl(x+\frac{1}{2}\biggr).
\end{equation*}
Using the formulas
\begin{equation*}
\ln \Gamma(x+1) -\biggl(x+\frac{1}{2}\biggr) \ln x+x-\ln \sqrt{2\pi}\,=\sum_{m=1}^\infty \frac{B_{2m}}{2m(2m-1) x^{2m-1}}
\end{equation*}
and
\begin{equation*}
\psi'(x) =\frac{1}{x}+\frac{1}{2x^{2}}+\sum_{m=1}^\infty \frac{B_{2m}}{x^{2m+1}}=\sum_{m=1}^{\infty} \frac{B_{m-1}}{x^{m}},
\end{equation*}
see~\cite[p.~257, 6.1.40]{aa} and~\cite[p.~260, 6.4.11]{aa}, figures out
\begin{equation}\label{h(x)-diff}
h(x) =\sum_{m=1}^\infty \frac{B_{2m}}{2m(2m-1)x^{2m-1}}-\sum_{m=1}^\infty \frac{B_{m-1}}{12\bigl(x+\frac{1}{2}\bigr)^{m}},
\end{equation}
where $B_{k}$ for $k\ge0$ denote Bernoulli numbers which may be generated by
\begin{equation*}
\frac{z}{e^z-1}=\sum_{k=0}^\infty B_k\frac{z^k}{k!}=1-\frac{z}2+\sum_{k=1}^\infty B_{2k}\frac{z^{2k}}{(2k)!}, \quad \vert z\vert<2\pi.
\end{equation*}
Making use of
\begin{multline*}
\sum_{k=1}^{m}\frac{a_{k}}{\bigl( x+\frac{1}{2}\bigr) ^{k}}
=\sum_{k=1}^{m}a_{k}\biggl( 1+\frac{1}{2x}\biggr) ^{-k}\frac{1}{x^{k}}\\
=\sum_{k=1}^{m}a_{k}\Biggl[\sum_{i=0}^\infty\binom{-k}{i} \frac{1}{2^{i}x^{i}}\Biggr]\frac{1}{x^{k}}
=\sum_{k=1}^{m}\sum_{i=0}^\infty \frac{a_{k}}{2^{i}}\binom{-k}{i}\frac{1}{x^{k+i}}
\end{multline*}
in~\eqref{h(x)-diff}, where $a_k$ is any sequence and
\begin{equation*}
\binom{-k}{i}=\frac{1}{i!}\prod_{\ell=0}^{i-1}(-k-\ell),
\end{equation*}
we obtain that
\begin{equation*}
h(x) =\frac{1}{240}\frac1{x^{3}}-\frac{11}{6720}\frac1{x^{5}}+\frac{107}{80640}\frac1{x^{7}} -\frac{2911}{1520640}\frac1{x^{9}}+ O\biggl(\frac1{x^{11}}\biggr).
\end{equation*}
The proof of Theorem~\ref{Gamma-asymp-thm} is complete.
\end{proof}

\begin{thm}\label{ii-thm-continou}
For every integer $x\ge 1$, we have
\begin{equation}\label{ii-continuous}
\exp \biggl(\frac{1}{240x^{3}}-\frac{11}{6720x^{5}}\biggr)
<\frac{e^{x}\Gamma(x+1)}{x^{x}\sqrt{2\pi x}\exp \bigl( \frac{1}{12}\psi'\bigl(x+\frac{1}{2}\bigr) \bigr)}
<\exp \frac{1}{240x^{3}}.
\end{equation}
\end{thm}

\begin{proof}
Let $f(x)$ and $g(x)$ for $x\in[1,\infty)$ be defined by~\eqref{f} and~\eqref{g} respectively.
Making use of inequalities
\begin{align*}
\frac{1}{12x}-\frac{1}{360x^{3}}+\frac{1}{1260x^{5}}
&<\ln \Gamma(x+1) -\biggl(x+\frac{1}{2}\biggr) \ln x+x-\frac{1}{2}\ln (2\pi)  \\
&<\frac{1}{12x}-\frac{1}{360x^{3}}+\frac{1}{1260x^{5}}-\frac{1}{1680x^{7}}
\end{align*}
and
\begin{equation*}
\frac{1}{x}+\frac{1}{2x^{2}}+\frac{1}{6x^{3}}-\frac{1}{30x^{5}}
<\psi'(x)
<\frac{1}{x}+\frac{1}{2x^{2}}+\frac{1}{6x^{3}}-\frac{1}{30x^{5}}+\frac{1}{42x^{7}},
\end{equation*}
which may be deduced from~\cite[Theorem~2 and Corollary~1]{Koumandos-jmaa-06}, finds that $f(x)<a(x)$ and $g(x)>b(x)$, where
\begin{align*}
a(x)  &=\frac{1}{12x}-\frac{1}{360x^{3}}+\frac{1}{1260x^{5}} -\frac{1}{12}\Biggl[ \frac{1}{x+\frac{1}{2}}+\frac{1}{2\bigl(x+\frac{1}{2}\bigr)^{2}}\\
&\quad +\frac{1}{6\bigl(x+\frac12\bigr)^{3}}-\frac{1}{30\bigl(x+\frac{1}{2}\bigr)^{5}}\Biggr] -\frac{1}{240x^{3}}\\
&= -\frac{A(x-1)}{5040x^{5}(2x+1)^{5}}\\
&<0,\\
b(x)  &=\frac{1}{12x}-\frac{1}{360x^{3}}+\frac{1}{1260x^{5}}-\frac{1}{1680x^{7}} -\frac{1}{12}\Biggl[ \frac{1}{x+\frac{1}{2}}+\frac{1}{2\bigl(x+\frac{1}{2}\bigr)^{2}}\\
&\quad +\frac{1}{6\bigl(x+\frac12\bigr)^{3}}-\frac{1}{30\bigl(x+\frac{1}{2}\bigr)^{5}} +\frac{1}{42\bigl(x+\frac{1}{2}\bigr)^{7}}\Biggr] -\frac{1}{240x^{3}}+\frac{11}{6720x^{5}}\\
&= \frac{B(x-1)}{20160x^{7}(2x+1)^{7}}\\
&>0,\\
A(x)&=3760x+6565x^{2}+5310x^{3}+1980x^{4}+264 x^{5}+785,\\
B(x)  &= 93268x+263179x^{2}+382830x^{3} \\
&\quad+315336 x^{4}+147504x^{5}+35952x^{6}+3424x^{7}+12547.
\end{align*}
The proof of Theorem~\ref{ii-thm-continou} is thus complete.
\end{proof}

\subsection*{Acknowledgements}
The work of the first author was supported in part by the Romanian National Authority for Scientific Research, CNCS-UEFISCDI, under Grant No. PN-II-ID-PCE-2011-3-0087.

\end{document}